\documentclass[reqno]{amsart}
\usepackage{geometry}
\usepackage{amsmath}
\usepackage{amsthm}
\usepackage{amsfonts}
\usepackage{amssymb}
\usepackage{tikz}
\usepackage{tikz,fullpage}
\usetikzlibrary{arrows,petri,topaths}
\usepackage{tkz-berge}
\usepackage[position=top]{subfig}
\usepackage{hyperref}
\usepackage{braket}
\begin{document}

\title{Optimization of the Implicit Constant for Upper Bounds for Moments of the Riemann Zeta Function}
\author{Tingyu Tao}
\maketitle

\newtheorem{DL}{Theorem}[section]
\newtheorem{MT}[DL]{Proposition}
\newtheorem{TL}[DL]{Corollary}
\newtheorem{YL}[DL]{Lemma}

\newcommand{\bi}{{\beta_i}}
\newcommand{\bj}{{\beta_j}}
\newcommand{\bI}{{\beta_{\mathcal I}}}
\newcommand{\m}{{\rm meas}}

\begin{abstract}
We optimized the implicit constant for the refined upper bounds for moments of the Riemann zeta-function proved by Harper. We also computed the implicit constant for the upper bounds for moments of the Riemann zeta-function proved by Soundararajan under certain conditions.
\end{abstract}

\section{Introduction}

Computing moments of the Riemann zeta function is a well-studied subject in number theory. For $k\geq 0$, the $2k$-th moment of the Riemann zeta function is defined as
$$I_k\left(T\right):=\int_0^T\left|\zeta\left(\frac{1}{2}+it\right)\right|^{2k}dt.$$
Hardy and Littlewood \cite{hardy} proved that $I_1\left(T\right)\sim T\log T$, and Ingham \cite{ingham} proved that $I_2\left(T\right)\sim\frac{1}{2\pi^2}T\left(\log T\right)^2$, while the asymptotic formulas of the $2k$-th moment for $k>2$ remain unproved. It is conjectured that for any $k\geq 0$, $I_k\left(T\right)\sim c_k T\left(\log T\right)^{k^2}$ for some constant $c_k$ that depends on $k$. Keating and Snaith \cite{keatingsnaith}, using random matrix theory, conjectured that for any $k\geq 0$, $c_k=a_kf_k$, where
$$f_k=\lim_{N\to\infty}N^{-k^2}\prod_{j=1}^N \frac{\Gamma\left(j\right)\Gamma\left(j+2k\right)}{\left(\Gamma\left(j+k\right)\right)^2},$$
and
$$a_k=\prod_{p}\left(1-\frac{1}{p}\right)^{k^2}\sum_{m=0}^\infty \left(\frac{\Gamma\left(m+k\right)}{m!\Gamma\left(k\right)}\right)^2\frac{1}{p^m}
,$$
where the product if over primes, and when $k$ is large enough, $c_k=O\left(e^{-k^2\log k}\right)$.

Although asymptotic results for the $2k$-th moment when $k>2$ remain open, some sharp bounds for moments have been proved. For lower bounds, Radziwiłł and Soundararajan \cite{lowerbound1} proved that $I_k\left(T\right)\gg_k T\left(\log T\right)^{k^2}$ when $k>1$. Heap and Soundararajan \cite{lowerbound2} proved that the same lower bound also holds when $0<k<1$. For upper bounds, Heap, Radziwiłł and Soundararajan \cite{sharpupper} proved that $I_k\left(T\right)\ll T\left(\log T\right)^{k^2}$ for $0\leq k\leq 2$ unconditionally. Soundararajan \cite{upperbound} proved that $I_k\left(T\right)\ll_k T\left(\log T\right)^{k^2+\varepsilon}$ for any $k\geq 0$ and any arbitrary $\varepsilon>0$, conditionally on the Riemann Hypothesis. Later, Harper \cite{harper} improved the upper bound to $T\left(\log T\right)^{k^2}$. More precisely, Harper proved that for a fixed $k\geq 0$, when $T$ is large enough, we have
\begin{equation}
\label{e1}
\int_T^{2T} \left|\zeta\left(\frac{1}{2}+it\right)\right|^{2k}dt\leq C\left(k\right) T\left(\log T\right)^{k^2},
\end{equation}
where $C\left(k\right)$ is a constant depending on $k$. Harper also discussed in the paper that $C\left(k\right)= e^{e^{O\left(k\right)}}$. The main focus of this paper is to build on Harper's proof and obtain an explicit constant value of $C\left(k\right)$, by optimizing various steps in the proof. We will prove the following result.

\begin{DL}
\label{theorem1}
Assume the Riemann Hypothesis is true, and let $k\geq 0$ be fixed. For large $T$, we have
$$\int_T^{2T}\left|\zeta\left(\frac{1}{2}+it\right)\right|^{2k}dt\ll e^{e^{18.63k}} T\left(\log T\right)^{k^2}.$$
where the implicit constant is absolute.
\end{DL}

That is, the constant $C\left(k\right)$ in (\ref{e1}) satisfies $C\left(k\right)=O\left(e^{e^{18.63k}}\right)$, and the implicit constant here does not depend on $k$.

\section{Proof of the Main Theorem}

Since $\int_{T}^{2T}\left|\zeta\left(\frac{1}{2}+it\right)\right|^{2k}dt=\int_{T}^{2T}e^{2k\log\left|\zeta\left(\frac{1}{2}+it\right)\right|}dt$, it is very helpful if we can obtain an approximation of $\log\left|\zeta\left(\frac{1}{2}+it\right)\right|$. Soundararajan \cite{upperbound} proved the following result.

\begin{MT}
\label{logap}
Assume the Riemann Hypothesis is true, and let $T$ be large. For any $2\leq x\leq T^2$, and any $T\leq t\leq 2T$, we have
$$\log\left|\zeta\left(\frac{1}{2}+it\right)\right|\leq {\Re}\left(\sum_{p\leq x}\frac{1}{p^{\frac{1}{2}+\frac{1}{\log x}+it}}\frac{\log\left(x/p\right)}{\log x}+\sum_{p\leq\min\left(\sqrt{x},\log T\right)}\frac{1/2}{p^{1+2it}}\right)+\frac{\log T}{\log x}+N$$
where $p$ denotes primes, and $N$ is an absolute constant.
\end{MT}

The main idea of Harper's proof of (\ref{e1}) is to find a partition of the interval $\left[T,2T\right]$, and compute the integral of $\left|\zeta\left(\frac{1}{2}+it\right)\right|^{2k}$ for each part. In order to partition the interval $\left[T,2T\right]$, we first introduce the following sequence $\left(\beta_j\right)$:

$$\beta_0:=0,\quad\quad \bi:=\frac{{c_1}^{i-1}}{\left(\log\log T\right)^2}\quad  \forall i\geq 1,$$
where $c_1>0$ is an absolute constant, and 
$$\mathcal I=\mathcal I_{k,T}:=1+\max\Set{i:\bi\leq e^{-c_2k}},$$
where $c_2$ is an absolute constant. For each $1\leq i\leq j\leq \mathcal I$, let
$$G_{\left(i,j\right)}\left(t\right)=G_{\left(i,j\right),T}\left(t\right):=\sum_{T^{\beta_{i-1}}<p\leq T^{\bi}}\frac{1}{p^{\frac{1}{2}+\frac{1}{\bj\log T}+it}}\frac{\log\left(T^{\bj}/p\right)}{\log \left(T^{\bj}\right)}.$$
The sets that form the partition are
$$\mathcal T=\mathcal T_{k,T}:=\Set{T\leq t\leq 2T:\left|{\Re}\sum_{T^{\beta_{i-1}}<p\leq T^{\bi}}\frac{1}{p^{\frac{1}{2}+\frac{1}{\bI\log T}+it}}\frac{\log\left(T^{\bI}/p\right)}{\log \left(T^{\bI}\right)}\right|\leq \bi^{-c_3},\ \forall 1\leq i\leq\mathcal I},$$
where $0<c_3<1$ is an absolute constant, and for all $0\leq j\leq\mathcal I-1$,
\begin{align*}
\mathcal S\left(j\right)=\mathcal S_{k,T}\left(j\right):=&\left\{T\leq t\leq 2T:\left|{\Re}G_{i,l}\left(t\right)\right|\leq\bi^{-c_3},\ \forall 1\leq i\leq j,\ \forall i\leq l\leq \mathcal I\right.\\
&\ \left.\text{but } \left|{\Re}G_{\left(j+1,l\right)}\left(t\right)\right|>{\beta_{j+1}}^{-c_3}\text{ for some } j+1\leq l\leq\mathcal I\right\}.
\end{align*}

Based on the above definition, we have
$$\left[T,2T\right]=\mathcal T\cup \bigcup_{j=0}^{\mathcal I-1}S\left(j\right),$$
so we have
\begin{equation}
\label{e2}
\int_T^{2T} \left|\zeta\left(\frac{1}{2}+it\right)\right|^{2k}dt=\int_{t\in \mathcal T}\left|\zeta\left(\frac{1}{2}+it\right)\right|^{2k}dt+\sum_{j=0}^{\mathcal I-1}\int_{t\in\mathcal S\left(j\right)}\left|\zeta\left(\frac{1}{2}+it\right)\right|^{2k}dt.
\end{equation}

For the integral on each individual part, we have the following lemmas.

\begin{YL}
\label{lemma1}
Let $k\geq 1$, and $T\geq e^{e^{\left(10000k\right)^2}}$. Denote $a=c_1^{1-c_3}$. If we have
$$\frac{ka^2}{a-1}e^{c_2k\left(1-c_3\right)}<\frac{1}{4},$$
then
$$\int_{t\in\mathcal T}\exp\left(2k{\Re}\sum_{p\leq T^{\bI}}\frac{1}{p^{\frac{1}{2}+\frac{1}{\bI\log T}+it}}\frac{\log\left(T^{\bI}/p\right)}{\log\left(T^{\bI}\right)}\right)dt\leq C_1 T\left(\log T\right)^{k^2},$$
where $C_1$ is an absolute constant.
\end{YL}

\begin{YL}
\label{lemma2}
Let $k\geq 1$, and $T\geq e^{e^{\left(10000k\right)^2}}$. Then
$$\m\left(\mathcal S\left(0\right)\right)\leq M Te^{-2\left(\log\log T\right)^2/c_1}$$
for some absolute constant $M$. Assume
$$\log c_1<\frac{c_1}{2},\quad \frac{c_2}{c_1\left(\frac{c_1}{4c_3-2}+1\right)}>2,$$
and $\forall 1\leq j\leq \mathcal I-1$, we have 
$$\mathcal I-j\leq\frac{1/\bj}{\log c_1}.$$
Let $c_4=\frac{c_1}{4c_3-2}+1$, then we have
$$\int_{t\in\mathcal S\left(j\right)}\exp\left(2k{\Re}\sum_{p\leq T^{\bj}}\frac{1}{p^{\frac{1}{2}+\frac{1}{\bj\log T}+it}}\frac{\log\left(T^{\bj}/p\right)}{\log\left(T^{\bj}\right)}\right)dt\leq C_2 e^{-\beta_{j+1}^{-1}\log\left(1/\beta_{j+1}\right)/c_4}T\left(\log T\right)^{k^2},$$
where $C_2$ is an absolute constant.
\end{YL}

\begin{YL}
\label{lemma3}
Let $k\geq 1$ and $T\geq e^{e^{\left(10000k\right)^2}}$. If the same assumptions in Lemma \ref{lemma1} are true, then
$$\int_{t\in\mathcal T}\exp\left(2k{\Re}\sum_{p\leq T^{\bI}}\frac{1}{p^{\frac{1}{2}+\frac{1}{\bI\log T}+it}}\frac{\log\left(T^{\bI}/p\right)}{\log\left(T^{\bI}\right)}+\sum_{p\leq \log T}\frac{1/2}{p^{1+2it}}\right)dt\leq D_1\left(k\right) T\left(\log T\right)^{k^2},$$
where $D_1\left(k\right)$ is a constant depending on $k$.
\end{YL}

\begin{YL}
\label{lemma4}
Let $k\geq 1$ and $T\geq e^{e^{\left(10000k\right)^2}}$. If the same assumptions in Lemma \ref{lemma2} are true, then for $1\leq j\leq\mathcal I-1$, we have
\begin{align*}
\int_{t\in\mathcal S\left(j\right)}\exp\left(2k{\Re}\sum_{p\leq T^{\bj}}\frac{1}{p^{\frac{1}{2}+\frac{1}{\bj\log T}+it}}\frac{\log\left(T^{\bj}/p\right)}{\log\left(T^{\bj}\right)}+\sum_{p\leq \log T}\frac{1/2}{p^{1+2it}}\right)dt\\
\leq D_2\left(k\right) e^{-\beta_{j+1}^{-1}\log\left(1/\beta_{j+1}\right)/c_4}T\left(\log T\right)^{k^2},
\end{align*}
where $D_2\left(k\right)$ is a constant depending on $k$.
\end{YL}

Now we prove the main theorem. 

\begin{proof}
[Proof of Theorem \ref{theorem1}] Let $x=T^{\bI}$ in Proposition \ref{logap}, so by Lemma \ref{lemma3} we have
$$\int_{t\in\mathcal T}\left|\zeta\left(\frac{1}{2}+it\right)\right|^{2k}dt\leq e^{2k/\bI+2kN}D_1\left(k\right)T\left(\log T\right)^{k^2}.$$

Let $x=T^{\bj}$ in Proposition \ref{logap}, so by Lemma \ref{lemma4} we get, for all $1\leq j\leq\mathcal I-1$,
$$\int_{t\in\mathcal S\left(j\right)}\left|\zeta\left(\frac{1}{2}+it\right)\right|^{2k}dt\leq e^{2k/\bj+2kN}D_2\left(k\right) e^{-\beta_{j+1}^{-1}\log\left(1/\beta_{j+1}\right)/c_4}T\left(\log T\right)^{k^2}.$$
By the restriction of $c_4$ in Lemma \ref{lemma2}, using geometric series, we have for $1\leq j\leq\mathcal I-1$
$$\sum_{j=1}^{\mathcal I-1}e^{2k/\bj}e^{-\beta_{j+1}^{-1}\log\left(1/\beta_{j+1}\right)/c_4}\leq 1.$$
Then we have
$$\sum_{j=1}^{\mathcal I-1}\int_{t\in\mathcal S\left(j\right)}\left|\zeta\left(\frac{1}{2}+it\right)\right|^{2k}dt\leq \mathcal I e^{2kN}D_2\left(k\right)T\left(\log T\right)^{k^2}.$$
When $j=0$, we have
$$\int_{t\in\mathcal S\left(0\right)}\left|\zeta\left(\frac{1}{2}+it\right)\right|^{2k}dt\leq\sqrt{\m\left(\mathcal S\left(0\right)\right)\int_{T}^{2T}\left|\zeta\left(\frac{1}{2}+it\right)\right|^{4k}dt}.$$
By Lemma 2 in \cite{upperbound}, we know for any $k\geq 1$ and $s>0$, there exists a constant $D_3\left(k,s\right)$ such that
$$\int_T^{2T} \left|\zeta\left(\frac{1}{2}+it\right)\right|^{2k}dt\leq D_3\left(k,s\right) T\left(\log T\right)^{k^2+s}.$$
Choose $s=64k^2$, then by Lemma \ref{lemma2}, when $c_1<100$, we have
\begin{align*}
\int_{t\in\mathcal S\left(0\right)}\left|\zeta\left(\frac{1}{2}+it\right)\right|^{2k}dt&\leq \sqrt{D_3\left(k,s\right)M} T \sqrt{\left(\log T\right)^{65k^2-2\log\log T/{c_1}}}\\
&\leq\sqrt{D_3\left(k,64k^2\right)M} T\left(\log T\right)^{k^2}.
\end{align*}
Let $D_3\left(k\right)=D_3\left(k,64k^2\right)$, then we have
$$\int_{t\in\mathcal S\left(0\right)}\left|\zeta\left(\frac{1}{2}+it\right)\right|^{2k}dt\leq\sqrt{D_3\left(k\right)M}T\left(\log T\right)^{k^2}.$$
Putting everything together, we have
$$\int_T^{2T} \left|\zeta\left(\frac{1}{2}+it\right)\right|^{2k}dt\leq \left(e^{2k/\bI+2kN}D_1\left(k\right)+\mathcal I e^{2kN}D_2\left(k\right)+\sqrt{D_3\left(k\right)M}\right)T\left(\log T\right)^{k^2}.$$
So, if we let 
$$C\left(k\right)=e^{2k/\bI+2kN}D_1\left(k\right)+\mathcal I e^{2kN}D_2\left(k\right)+\sqrt{D_3\left(k\right)M},$$
then (\ref{e1}) is proved.

Harper proved in section 6 of \cite{harper} that $D_1\left(k\right)$ and $D_2\left(k\right)$ are both of size $e^{O\left(k\right)}$. We have $D_3\left(k\right)=O\left(k\right)$ by (\ref{d3}). From the definition of $\bI$ we know that $\bI\leq c_1e^{-c_2k}=O\left(e^{-c_2k}\right)$, hence $e^{2k/\bI}=O\left(e^{e^{c_2k}}\right)$, which becomes the dominating term in $C\left(k\right)$. Hence we can conclude that $C\left(k\right)=O\left(e^{e^{c_2k}}\right)$, which means that the size of the implicit constant depends entirely on the choice of $c_2$. We will determine the choice of $c_2$ in Section 4 to finish the proof.
\end{proof}

Now we prove that $D_3=O\left(k^2\right)$. We first introduce the following theorem from \cite{upperbound}.

\begin{DL}
\label{theorem2}
Assume the Riemann Hypothesis is true, let $T$ be large, and $V\geq 3$. Define $S\left(T,V\right):=\Set{T\leq t\leq 2T|\log\left|\zeta\left(\frac{1}{2}+it\right)\right|\geq V}$. Denote $\log\log\log T$ as $\log_3 T$. If $10\sqrt{\log\log T}\leq V\leq \log\log T$, then there exists $b_1>0$ such that
$$\m\left(S\left(T,V\right)\right)\leq b_1\cdot T\frac{V}{\sqrt{\log\log T}}\exp\left(-\frac{V^2}{\log\log T}\left(1-\frac{4}{\log_3 T}\right)\right).$$
If $\log\log T<V\leq\frac{1}{2}\log\log T\log_3 T$, then there exists $b_2>0$ such that
$$\m\left(S\left(T,V\right)\right)\leq b_2\cdot T\frac{V}{\sqrt{\log\log T}}\exp\left(-\frac{V^2}{\log\log T}\left(1-\frac{7V}{4\log\log T\log_3 T}\right)^2\right).$$
If $\frac{1}{2}\log\log T\log_3 T<V$, then there exists $b_3>0$ such that
$$\m\left(S\left(T,V\right)\right)\leq b_3\cdot T\exp\left(-\frac{1}{33} V\log V\right).$$
\end{DL}

To find $D_3\left(k\right)$, firstly note that
$$\int_T^{2T} \left|\zeta\left(\frac{1}{2}+it\right)\right|^{2k}dt=-\int_{-\infty}^\infty e^{2kV}d\m\left(S\left(T,V\right)\right)=2k\int_{-\infty}^\infty e^{2kV}\m\left(S\left(T,V\right)\right)dV.$$

We can separate the integral $\int_{-\infty}^\infty e^{2kV}\m\left(S\left(T,V\right)\right)dV$ into four parts depending on the size of $V$. For simplicity, we denote $T_1=10\sqrt{\log\log T}$, $T_2=\log\log T$, $T_3=\frac{1}{2}\log\log T\log_3 T$. Since $e^{2kV}\m\left(S\left(T,V\right)\right)$ is non-negative, we have
\begin{align*}
\int_{-\infty}^\infty e^{2kV}\m\left(S\left(T,V\right)\right)dV=&\int_{-\infty}^{T_1} e^{2kV}\m\left(S\left(T,V\right)\right)dV+\int_{T_1}^{T_2}e^{2kV}\m\left(S\left(T,V\right)\right)dV\\
&+\int_{T_2}^{T_3}e^{2kV}\m\left(S\left(T,V\right)\right)dV+\int_{T_3}^{\infty}e^{2kV}\m\left(S\left(T,V\right)\right)dV.
\end{align*}

For the first integral, we have
$$\int_{-\infty}^{T_1} e^{2kV}\m\left(S\left(T,V\right)\right)dV\leq \int_{-\infty}^{T_1} e^{2kV} TdV=\frac{T}{2T_1} e^{2T_1k}.$$

When $T\geq e^{e^{\left(10000k\right)^2}}$, $T\geq e^{e^{400}}$, so we have $10\sqrt{\log\log T}\leq \log\log T/2$, then
$$\frac{T}{2T_1}e^{2T_1k}\leq Te^{\log\log Tk}=T\left(\log T\right)^k.$$

\newcommand{\erfc}{{\rm erfc}}
\newcommand{\erf}{{\rm erf}}

For the second integral, by Theorem \ref{theorem2}, we have
\begin{align*}
&\int_{T_1}^{T_2}e^{2kV}\m\left(S\left(T,V\right)\right)dV\leq b_1\int_{T_1}^{T_2} e^{2kV} T\frac{V}{\sqrt{\log\log T}}\exp\left(-\frac{V^2}{\log\log T}\left(1-\frac{4}{\log_3 T}\right)\right)dV\\
&=\frac{b_1T\sqrt{\log\log T}\log_3 T}{2\log T\sqrt{1-\frac{4}{\log_3T}}\left(\log_3 T-4\right)}\\
&\left(-k\sqrt{\pi}\erf\left(\frac{\sqrt{\log\log T}\left(4+\left(k-1\right)\log_3 T\right)}{\sqrt{\log_3 T-4}\sqrt{\log_3 T}}\right)\left(\log T\right)^{1+\frac{k^2\log_3T}{\log_3 T-4}}\sqrt{\log\log T}+\right.\\
&+k\sqrt{\pi}\erf\left(\frac{40+\left(k\sqrt{\log\log T}-10\right)\log_3 T}{\sqrt{\log_3 T-4}\sqrt{\log_3 T}}\right)\left(\log T\right)^{1+\frac{k^2\log_3 T}{\log_3 T-4}}\sqrt{\log\log T}\\
&+\left.\left(\exp\left(20k\sqrt{\log\log T}-100+\frac{400}{\log_3 T}\right)\log T-e^{100}\left(\log T\right)^{2k+\frac{4}{\log_3 T}}\right)\sqrt{1-\frac{4}{\log_3 T}}\right)\\
&\leq b_1 T\sqrt{\log\log T}\left(k\sqrt{\pi}\left(\log T\right)^{\frac{k^2\log_3 T}{\log_3T-4}}\sqrt{\log\log T}+e^{20k\sqrt{\log\log T}}\right).
\end{align*}

In the above inequality, the function $\erf$ is the error function defined by
$$\erf \left(x\right)=\frac{2}{\pi}\int_0^x e^{-t^2} dt.$$
which is bounded by $-1$ and $1$.

When $T\geq e^{e^{\left(10000k\right)^2}}$, $\log_3 T>5$, hence $\frac{4k^2}{\log_3 T-4}<\frac{64k^2}{2}$, $\log\log T<\left(\log T\right)^{64k^2/2}$, and $20\sqrt{\log\log T}\leq \log\log T$ then we have
\begin{align*}
\int_{T_1}^{T_2}e^{2kV}\m\left(S\left(T,V\right)\right)dV&\leq b_1k \sqrt{\pi} T\left(\log T\right)^{65k^2}+b_1 T\sqrt{\log\log T}\left(\log T\right)^k\\
&\leq 2b_1k\sqrt{\pi} T\left(\log T\right)^{65k^2}.
\end{align*}

For the third integral, since on $\left[T_2,T_3\right]$, $V\leq \frac{1}{2}\log\log T\log_3T$, so $\frac{7V}{4\log\log T\log_3T}\leq\frac{7}{8}$, we have

\begin{align*}
\int_{T_2}^{T_3}e^{2kV}\m\left(S\left(T,V\right)\right)dV\leq& b_2\int_{T_2}^{T_3} e^{2kV}  T\frac{V}{\sqrt{\log\log T}}\exp\left(-\frac{V^2}{64\log\log T}\right)dV\\
=&32 b_2T\sqrt{\log\log T}\left(\left(\log T\right)^{-\frac{1}{64}+2k}-\left(\log T\right)^{k\log_3T-\frac{1}{256}\left(\log_3T\right)^2}\right.\\
&+8k\sqrt{\pi}\left(-\erf\left(\frac{1}{8}\left(1-64k\right)\sqrt{\log\log T}\right)\right.\\
&+\left.\erf\left(\frac{1}{16}\sqrt{\log\log T}\left(-128k+\log_3 T\right)\right)\right)\\
&\left.\left(\log T\right)^{64k^2}\sqrt{\log\log T}\right)\\
\leq &256b_2k\sqrt{\pi}T\log\log T\left(\log T\right)^{64k^2}\\
\leq &256b_2k\sqrt{\pi} T\left(\log T\right)^{65k^2}.
\end{align*}

For the fourth integral, by Theorem \ref{theorem2}, we have

$$\int_{T_3}^\infty e^{2kV}\m\left(S\left(T,V\right)\right)dV\leq b_3\int_{T_3}^\infty e^{2kV}TV^{-V/33}dV.$$

When $V\leq e^{33\left(2k+1\right)}$, we have

\begin{align*}
\int_{T_3}^{e^{33\left(2k+1\right)}}e^{2kV}TV^{-V/33}& \leq \frac{b_3}{132k}T\left(\log\log T\log_3T\right)^2\left(\log T\right)^{k\log_3T}\\
&\leq \frac{b_3}{132k}T\left(\log T\right)^{11k^2}.
\end{align*}

When $V\geq e^{33\left(2k+1\right)}$, we have
$$e^{2kV}TV^{-V/33}\leq e^{2kV}T\left(e^{33\left(2k+1\right)}\right)^{-V/33}=Te^{-V},$$
hence we have
$$\int_{e^{33\left(2k+1\right)}}^\infty e^{2kV}\m\left(S\left(T,V\right)\right)dV< b_3 T.$$

Combining the above four results, we have

\begin{align*}
&\int_{T}^{2T} \left|\zeta\left(\frac{1}{2}+it\right)\right|^{2k}dt\\
\leq &2k\left(T\left(\log T\right)^k+2b_1k\sqrt{\pi} T\left(\log T\right)^{65k^2}+256b_2k\sqrt{\pi}T\left(\log T\right)^{65k^2}+\frac{b_3}{132k}T\left(\log T\right)^{11k^2}+b_3 T\right)\\
\leq &\left(2+2b_1\sqrt{\pi}+512b_2\sqrt{\pi}+3b_3\right)k^2 T\left(\log T\right)^{65k^2}.
\end{align*}

So, we have
\begin{equation}
\label{d3}
D_3\left(k\right)\leq \left(2+b_1\frac{8}{\varepsilon}\sqrt{\pi}+512b_2\sqrt{\pi}+3b_3\right)k^2=O\left(k^2\right).
\end{equation}

\section{Proofs of the Lemmas}

\begin{proof}
[Proof of Lemma \ref{lemma1}]
Denote $G_{\left(i,\mathcal I\right)}\left(t\right)$ as $F_i\left(t\right)$, then by the definition of the set $\mathcal T$, we know that $\left|{\Re}F_i\left(t\right)\right|\leq\bi^{-c_3}$ for all $t\in\mathcal T$. Then
$$\int_{t\in\mathcal T}\exp\left(2k{\Re}\sum_{p\leq T^{\bI}}\frac{1}{p^{\frac{1}{2}+\frac{1}{\bI\log T}+it}}\frac{\log\left(T^{\bI}/p\right)}{\log\left(T^{\bI}\right)}\right)dt=\int_{t\in\mathcal T}\prod_{i=1}^{\mathcal I}\exp\left(2k{\Re}F_i\left(t\right)\right)dt$$
\begin{align*}&=\int_{t\in\mathcal T}\prod_{i=1}^{\mathcal I}\left(1+O\left(e^{-100k\bi^{-c_3}}\right)\right)\left(\sum_{0\leq j\leq 100k\bi^{-c_3}}\frac{k{\Re} F_i\left(t\right)^j}{j!}\right)^2dt\\
&\leq \left(1+O\left(e^{-100k\bi^{-c_3}}\right)\right)\int_T^{2T} \prod_{i=1}^{\mathcal I}\left(\sum_{0\leq j\leq 100k\bi^{-c_3}}\frac{k{\Re} F_i\left(t\right)^j}{j!}\right)^2dt.
\end{align*}
For all $1\leq i\leq\mathcal I$ and $T\leq t\leq 2T$, we have
$${\Re} F_i\left(t\right)=\sum_{T^{\beta_{i-1}}<p\leq T^{\bi}}\frac{\cos\left(t\log p\right)}{p^{\frac{1}{2}+\frac{1}{\bI\log T}}}\frac{\log\left(T^{\bI}/p\right)}{\log \left(T^{\bI}\right)},$$
so
\begin{align*}
&\int_T^{2T} \prod_{i=1}^{\mathcal I}\left(\sum_{0\leq j\leq 100k\bi^{-c_3}}\frac{k{\Re} F_i\left(t\right)^j}{j!}\right)^2dt.\\
=&\sum_{\overline{j},\overline{l}}\prod_{i=1}^{\mathcal I}\frac{k^{j_i}k^{l_i}}{j_i!l_i!}\sum_{\overline{p},\overline{q}}C\left(\overline{p},\overline{q}\right)\int_T^{2T}\prod_{i=1}^{\mathcal I}\prod_{\substack{1\leq r\leq j_i,\\ 1\leq s\leq l_i}}\cos\left(t\log p\left(i,r\right)\right)\cos\left(t\log q\left(i,s\right)\right)dt
\end{align*}
where the outer sum is over vectors $\overline{j}=\left(j_1,\cdots,j_{\mathcal I}\right)$ and $\overline{l}=\left(l_1,\cdots,l_{\mathcal I}\right)$, whose entries are between $0$ and $100k\bi^{-\frac{3}{4}}$, and the inner sum is over vectors $\overline{p}=\left(p\left(1,1\right),p\left(1,2\right),\cdots,p\left(1,j_1\right),\cdots,p\left(\mathcal I,j_{\mathcal I}\right)\right)$, $\overline{q}=\left(q\left(1,1\right),\cdots,q\left(\mathcal I,l_{\mathcal I}\right)\right)$, whose entries are primes, such that for all $1\leq i\leq\mathcal I$,
$$T^{\beta_{i-1}}<p\left(i,1\right),\cdots,p\left(i,j_i\right),q\left(i,1\right),\cdots,q\left(i,l_i\right)\leq T^{\bi},$$
and
$$C\left(\overline{p},\overline{q}\right):=\prod_{i=1}^{\mathcal I}\prod_{\substack{1\leq r\leq j_i,\\ 1\leq s\leq l_i}}\frac{1}{p\left(i,r\right)^{\frac{1}{2}+\frac{1}{\bI\log T}}}{\log\left(T^{\bI}/p\left(i,r\right)\right)}{\log\left(T^{\bI}\right)}\frac{1}{q\left(i,s\right)^{\frac{1}{2}+\frac{1}{\bI\log T}}}{\log\left(T^{\bI}/q\left(i,s\right)\right)}{\log\left(T^{\bI}\right)}.$$
Then we have
\begin{equation}
\label{ratio}
\prod_{i=1}^{\mathcal I}\prod_{\substack{1\leq r\leq j_i,\\ 1\leq s\leq l_i}}p\left(i,r\right)q\left(i,s\right)\leq\prod_{i=1}^{\mathcal I} T^{\bi\left(j_i+l_i\right)}\leq \prod_{i=1}^{\mathcal I}T^{2\beta_{\mathcal I}^{1-c_4}}.
\end{equation}

In the assumption, we have
$$\frac{k\left(c_1^{1-c_3}\right)^2}{c_1^{1-c_3}-1}e^{-c_2k\left(1-c_3\right)}<\frac{1}{4}.$$

Then $2\bI^{1-c_4}<\frac{1}{2}$. Let $\delta=\frac{1}{2}-2\bI^{1-c_4}$, then
$$\prod_{i=1}^{\mathcal I}\prod_{\substack{1\leq r\leq j_i,\\ 1\leq s\leq l_i}}p\left(i,r\right)q\left(i,s\right)\leq T^{\frac{1}{2}-\delta}$$

For $n={p_1}^{\alpha_1}\cdots{p_r}^{\alpha_r}$, define $f\left(n\right)$ as: if any of $\alpha_i$ is odd, then $f\left(n\right)=0$, and otherwise
$$f\left(n\right):=\prod_{i=1}^{r}\frac{1}{2^{\alpha_i}}\frac{\alpha_i!}{\left(\alpha_i/2\right)!^2}.$$
Note that $f$ is a multiplicative, non-negative function that is supported on squares. Then we have
$$\int_T^{2T}\prod_{i=1}^{\mathcal I}\prod_{\substack{1\leq r\leq j_i,\\ 1\leq s\leq l_i}}\cos\left(t\log p\left(i,r\right)\right)\cos\left(t\log q\left(i,s\right)\right)dt=Tf\left(\prod_{i=1}^{\mathcal I}\prod_{\substack{1\leq r\leq j_i,\\ 1\leq s\leq l_i}}p\left(i,r\right)q\left(i,s\right)\right)+O\left(T^{0.1}\right).$$
Define
$$D\left(\overline{p},\overline{q}\right):=\prod_{i=1}^{\mathcal I}\prod_{\substack{1\leq r\leq j_i,\\ 1\leq s\leq l_i}}\frac{1}{\sqrt{p\left(i,r\right)}}\frac{1}{\sqrt{q\left(i,s\right)}}$$
then $C\left(\overline{p},\overline{q}\right)\leq D\left(\overline{p},\overline{q}\right)$. Hence
\begin{align}
\int_{t\in\mathcal T}\exp\left(2k{\Re}F_i\left(t\right)\right)dt\ll &T\sum_{\overline{j},\overline{l}}\prod_{i=1}^{\mathcal I}\frac{k^{j_i}k^{l_i}}{j_i!l_i!}\sum_{\overline{p},\overline{q}}D\left(\overline{p},\overline{q}\right)f\left(\prod_{i=1}^{\mathcal I}\prod_{\substack{1\leq r\leq j_i,\\ 1\leq s\leq l_i}}p\left(i,r\right)q\left(i,s\right)\right)\label{part1}\\
&+T^{\frac{1}{2}-\delta}\sum_{\overline{j},\overline{l}}\prod_{i=1}^{\mathcal I}\frac{k^{j_i}k^{l_i}}{j_i!l_i!}\sum_{\overline{p},\overline{q}}D\left(\overline{p},\overline{q}\right)\label{part2}
\end{align}

By counting squares, we can show that
$$T\sum_{\overline{j},\overline{l}}\prod_{i=1}^{\mathcal I}\frac{k^{j_i}k^{l_i}}{j_i!l_i!}\sum_{\overline{p},\overline{q}}D\left(\overline{p},\overline{q}\right)f\left(\prod_{i=1}^{\mathcal I}\prod_{\substack{1\leq r\leq j_i,\\ 1\leq s\leq l_i}}p\left(i,r\right)q\left(i,s\right)\right)\leq T\left(\log T\right)^{k^2},$$
and
$$T^{\frac{1}{2}-\delta}\sum_{\overline{j},\overline{l}}\prod_{i=1}^{\mathcal I}\frac{k^{j_i}k^{l_i}}{j_i!l_i!}\sum_{\overline{p},\overline{q}}D\left(\overline{p},\overline{q}\right)\leq T^{1-2\delta}e^{2k\mathcal I},$$
which is smaller than $T\left(\log T\right)^{k^2}$. Lemma \ref{lemma1} is then proved. 
\end{proof}

\begin{proof}
[Proof of Lemma \ref{lemma2}]
The idea of proving Lemma \ref{lemma2} is similar to that of Lemma \ref{lemma1}, which gives
\begin{align*}
&\int_{t\in\mathcal S\left(j\right)}\exp\left(2k\Re\sum_{p\leq T^{\bj}}\frac{1}{p^{\frac{1}{2}+\frac{1}{\bj\log T}+it}}\frac{\log\left(T^{\bj}/p\right)}{\log\left(T^{\bj}\right)}\right)dt\\
&\ll \left(\mathcal I-j\right)\exp\left(k^2\sum_{p\leq T^{\bj}}\frac{1}{p}\right)\left(\frac{\beta_{j+1}^{1/2}}{c_1}\sum_{T^{\bj}<p\leq T^{\beta_{j+1}}}\frac{1}{p}\right)^{2/c_1\beta_{j+1}}.
\end{align*}
When $j=0$, the left-hand-side is the measure of $\mathcal S\left(0\right)$, so
$$\m\left(\mathcal S\left(0\right)\right)\ll Te^{-2\left(\log\log T\right)^2/c_1},$$
and when $1\leq j\leq\mathcal I-1$, since 
\begin{equation}
\mathcal I-j\leq\frac{\log\left(1/\bj\right)}{\log c_1},\quad\quad \sum_{T^{bj}<p\leq T^{\beta_{j+1}}}\frac{1}{p}=\log c_1+o\left(1\right)\leq \frac{c_1}{2},\label{lemma2res}
\end{equation}
so
$$\int_{t\in\mathcal S\left(j\right)}\exp\left(2k{\Re}\sum_{p\leq T^{\bj}}\frac{1}{p^{\frac{1}{2}+\frac{1}{\bj\log T}+it}}\frac{\log\left(T^{\bj}/p\right)}{\log\left(T^{\bj}\right)}\right)dt\ll e^{-\beta_{j+1}^{-1}\log\left(1/\beta_{j+1}\right)/c_4}T\left(\log T\right)^{k^2}.$$
where
$$c_4=\frac{c_1}{4c_3-2}+1$$
For simplicity, denote $b=c_1\left(\frac{c_1}{4c_3-2}+1\right)$, then from the assumption, we have for all $1\leq j\leq\mathcal I-1$, $\frac{c_2}{b}>2$
hence
\begin{align*}
e^{2k/\bj}e^{-\beta_{j+1}^{-1}\log\left(1/\beta_{j+1}\right)/c_4}&\leq \exp\left(2k-\log\left(1/beta_{j+1}\right)/\left(b\bj\right)\right)\\
&\leq \exp\left(\left(2-\frac{c_2}{b}\right)/\bj\right)
\end{align*}
since $2-\frac{c_2}{b}<0$, so the sum of the above value for all $1\leq j\leq\mathcal I-1$ is bounded by an absolute constant.
\end{proof}

\section{Optimization}

By the previous discussion, among all the parameters, only $c_1,c_2,c_3$ affect the implicit constant, while $c_4$ can be written as an expression in terms of $c_1,c_2,c_3$. In order to simplify the calculation, we set
$$a={c_1}^{1-c_4},\quad b= c_1\left(\frac{c_1}{4c_3-2}+1\right),$$
and they need to satisfy the following relations:
$$\frac{c_2}{b}>2,$$
$$\frac{ka^2}{a-1}e^{-ak\left(1-c_4\right)}<\frac{1}{4}.$$

Since $C\left(k\right)=O\left(e^{e^{c_2k}}\right)$, so the objective is to make $c_2$ as small as possible, which, by $c_2/b>2$, requires $b$ is as small as possible too. Consider $c_1$ and $c_3$ as two independent variables, and then $b$ is a function with respect to $c_1$ and $c_3$. There is no obvious way to determine the minimum point directly, but we can find the numerical minimum by exhibiting a table of values of $b$ with different choices of $c_1$ and $c_3$. 

By the general setting, in order to make $\left(\beta_j\right)$ an increasing sequence, $c_1>1$. In order to make $c_4$ positive, $4c_3-2>0$, hence $c_3>\frac{1}{2}$. Also from the assumption we have $c_3<1$. First only consider the case $c_1<20$. Using mathematica, the minimum of $b$ in this case is reached when $c_1=1.38$ and $c_3=0.56$. If $c_1\geq 20$, then
$$b=c_1\left(\frac{c_1}{4c_3-2}+1\right)\geq c_1\left(\frac{c_1}{2}+1\right)\geq 220.$$
Clearly the minimum of $b$ cannot be reached when $c_1\geq 20$.

So, the minimum of $b$ is indeed reached when $c_1\approx 1.38$ and $c_3\approx 0.56$, which gives the value of $c_2$ as $18.63$. This shows that the minimum of the implicity constant is $e^{e^{18.63k}}$. Hence Theorem \ref{theorem1} is proved.

\end{document}